\documentclass{article}
\usepackage[utf8]{inputenc}
\usepackage{amssymb,amsmath,amsthm}
\usepackage[margin=1in]{geometry}
\usepackage{graphicx}
\graphicspath{ {./ProblemImages/} }
\usepackage{subfiles}
\usepackage[mathscr]{euscript}
\usepackage{enumitem}
\usepackage{multicol}
\usepackage{comment}

\usepackage{titlesec}
\usepackage{mathdots}
\usepackage{marginnote}
\usepackage{xparse}
\usepackage{tikz}

\usepackage{hyperref}
\hypersetup{
 colorlinks=true,
 linkcolor=blue,
 filecolor=magenta, 
 urlcolor=cyan,
}

\newcommand{\R}{\mathbb R}
\newcommand{\C}{\mathbb C}
\newcommand{\N}{\mathbb N}

\newcommand{\W}{\mathbb W}

\newcommand{\lv}{\left|}
\newcommand{\rv}{\right|}
\newcommand{\lV}{\left\|}
\newcommand{\rV}{\right\|}

\newcommand{\lp}{\left(}
\newcommand{\rp}{\right)}
\newcommand{\lc}{\left\{}
\newcommand{\rc}{\right\}}
\newcommand{\inv}{^{-1}}
\newcommand{\Th}{^{\text{th}}}
\newcommand{\ar}{\xrightarrow[]{}}

\newcommand{\bbm}[1]{\begin{bmatrix}#1\end{bmatrix}}

\newcommand{\mc}[1]{\mathcal{#1}}
\newcommand{\cs}[1]{\lc#1\rc}
\newcommand{\vs}[1]{\lv#1\rv}
\newcommand{\ps}[1]{\lp#1\rp}
\newcommand{\bs}[1]{\left[#1\right]}
\newcommand{\Vs}[1]{\lV#1\rV}

\newcommand{\mcdf}{\bbm{M&C\\D&F}}
\newcommand{\rdef}{M+C\ps{\lambda I-F}\inv D}
\newcommand{\li}{\lambda I}

\newcommand{\lifi}{\ps{\lambda I-F}\inv }
\newcommand{\red}{\mc{R}\ps{\lambda,S,A}}

\newcommand{\fixmewo}{\marginnote{FIX THIS LATER}}
\newcommand{\fixmew}[1]{\marginnote{FIX THIS LATER:\newline#1\newline FIX THIS LATER}}

\DeclareDocumentCommand\fixme{ g }{
\IfNoValueF{#1}{\fixmewo}{\fixmew{#1}}
}

\DeclareMathOperator{\tr}{tr}

\DeclareMathOperator{\col}{col}

\DeclareMathOperator{\proj}{proj}

\newtheorem{theorem}{Theorem}[section]
\newtheorem{corollary}[theorem]{Corollary}
\newtheorem{lemma}[theorem]{Lemma}

\theoremstyle{definition}
\newtheorem{definition}[theorem]{Definition}
\newtheorem{example}[theorem]{Example}
\newtheorem{remark}[theorem]{Remark}

\title{Isospectral reductions and quantum walks on graphs}
\author{Mark Kempton\thanks{Department of Mathematics, Brigham Young University, Provo UT (\texttt{mkempton@mathematics.byu.edu}).} and John Tolbert\thanks{Department of Mathematics, Wake Forest University, Winston-Salem NC  (\texttt{johntolbert2001@gmail.com}).}}
\date{}

\begin{document}

\maketitle

\begin{abstract}
    We give a new formula for computing the isospectral reduction of a matrix (and graph) down to a submatrix (or subgraph). Using this, we generalize the notion of isospectral reductions.  In addition, we give a procedure for constructing a matrix whose isospectral reduction down to a submatrix is given.  We also prove that the isospectral reduction completely determines the restriction of the quantum walk transition matrix to a subset. Using these, we construct new families of simple graphs exhibiting perfect quantum state transfer.
\end{abstract}



\section{Introduction}\label{sec:intro}

Isospectral reductions are a valuable recent tool for studying spectra of graphs and matrices.  Given a graph $G$ with adjacency matrix $A$, and a subset $S$ of its vertex set, the isospectral reduction of $A$ to $S$ is given by 
\[
\mc{R}(\lambda,S,A) = A_{S\times S}-A_{S\times\overline{S}}\ps{A_{\overline{S}\times\overline{S}}-\li}^{-1}A_{\overline{S}\times S}.
\]
Note that $\mathcal{R}(\lambda,S,A)$ is an $|S|\times |S|$ matrix with entries that are rational functions in the variable $\lambda$.  Isospectral reductions have the remarkable feature of yielding a smaller matrix that preserves the spectral properties of the original.  We give a more detailed development in Section \ref{sec:iso_red}.  See also \cite{MR3237552} for a detailed introduction.  Isospectral reductions have been used in a variety of contexts in graph theory and network science.  In particular, they have been used to improve the eigenvalue approximations of Gershgorin, Brauer, and Brualdi \cite{Gershgorin_1931,Brauer_1947,Brualdi_1982}; study the pseudo-spectra of graphs and matrices \cite{Vasquez_2015}; create stability preserving transformations of networks \cite{Bunimovich_2011,Bunimovich_2013,Reber_2019}; and study the survival probabilities in open dynamical systems \cite{bunimovich_2014}.

Notably, recent work \cite{kempton2020characterizing} has shown that isospectral reductions can be used to characterize when two vertices $u$ and $v$ of graph $G$ are \emph{cospectral}, i.e. when $G\backslash u$ and $G\backslash v$ have the same spectrum.  Cospectral vertices are of interest in understanding what spectral information can reveal about a graph.  In addition, much of the research around cospectral vertices is motivated by applications to quantum information theory.  Specifically, we say that there is \emph{perfect state transfer} from vertex $u$ to $v$ at time $t=\tau$ if the quantum walk transition matrix $e^{-itA}$ satisfies \[e^{-i\tau A}e_u=\gamma e_v\] for some unit complex number $\gamma$, where $e_x$ denotes the indicator vector for vertex $x$.  We will give more detailed preliminaries concerning quantum walks and perfect state transfer in Section \ref{sec:quantum_walk}. See also \cite{godsil_survey,kay2010perfect} for good surveys introducing this area.

It is well known (see \cite{godsil_survey,kay2010perfect}) that a necessary condition for perfect state transfer from $u$ to $v$ is that $u$ and $v$ be cospectral. Some recent work \cite{rontgen2020designing} has already investigated the question of using isospectral reductions to construct graphs with cospectral vertices as in \cite{kempton2020characterizing} and achieve approximate perfect state transfer with appropriate weights.  One goal of the present paper is to build off of the work of \cite{kempton2020characterizing} and establish an even stronger connection between the isospectral reduction and quantum state transfer phenomena.  Indeed, one of our main results (Theorem \ref{Quantum walk equivalence to reduction} below) is that the restriction of the quantum walk transition matrix to a subset of vertices is completely determined by the isospectral reduction to that subset.  In particular, whether and when perfect state transfer occurs between two vertices can be completely determined by the isospectral reduction to those two vertices.  This makes the question raised in \cite{kempton2020characterizing} of how to reverse engineer the isospectral reduction process of significant relevance to the quantum state transfer problem.  The other main contribution of this paper is to address this question and give a procedure for constructing a matrix with entries in $\R$ or $\C$ whose isospectral reduction is given.  We refer to this as \emph{unfolding} the isospectral reduction.

\subsection{Outline}

The remainder of this paper is organized as follows.  In Section \ref{sec:iso_red}, we begin with the basic preliminaries regarding isospectral reductions, and prove several new lemmas and technical results that will be useful later when investigating the unfolding of isospectral reductions.  In Section \ref{sec:new_form}, we give a new expression for isospectral reductions, namely
\begin{equation}\label{eq:new_form}
    \mc{R}(\lambda,S,A) = \lambda I-\left(\Sigma^T(\lambda I -A)^{-1}\Sigma\right)^{-1}
\end{equation} where $\Sigma=[I ~ 0]^T$ (see Theorem \ref{Formula 2 for the reduction} below).  This new formula will provide some necessary insight for several of the results later on.  First, in Section \ref{sec:gen_fun}, we investigate the combinatorial interpretation of isospectral reductions in terms of walk generating functions, and prove a relationship between the walk generating function relative to a subset, and the \emph{non-returning} walk generating function for that subset (Theorem \ref{thm:gen_funct}), generalizing a formula involving single node subsets from \cite{kempton2020characterizing}.  Next in Section \ref{sec:gen_red}, we use the formula from (\ref{eq:new_form}) to define \emph{generalized} isospectral reductions, replacing the matrix $\Sigma$ with any matrix with orthonormal columns, and deduce many of the classic results from the theory of isospectral reductions in the generalized case.  In Section \ref{sec:equit_part}, we show how the theory of equitable partitions of matrices can be viewed as a special case of these generalized isospectral reductions (Theorem \ref{thm:equit}).

In Section \ref{sec:unfolding}, we directly address the question of \emph{unfolding} an isospectral reduction, that is, constructing a matrix with real or complex entries with a given isospectral reduction.  In particular, we give necessary and sufficient conditions for a reduction to have a Hermitian unfolding, and give a procedure for finding such an unfolding (Theorem \ref{Unfolding Existence and Procedure}).

In Section \ref{sec:quantum_walk} we apply these ideas to the theory of quantum walks on graphs.  We first give the necessary preliminaries regarding quantum walks, perfect state transfer, and other tools, and then prove that the isospectral reduction completely determines the quantum walk matrix restricted to a subset (Theorem \ref{Quantum walk equivalence to reduction}).  The proof of this theorem relies heavily on the perspective given from the formula in (\ref{eq:new_form}).

Finally, in Section \ref{sec:examples}, we apply the quantum walk determination theorem of Section \ref{sec:quantum_walk}, along with the unfolding procedure of Section \ref{sec:unfolding}, to give a novel procedure to construct new examples of simple graphs exhibiting perfect state transfer.  The idea here is to start with a graph whose quantum walk behavior is known, isospectrally reduce down to a conveniently chosen subset, then unfold the result to a new graph that will have the same quantum walk behavior as the original.  Specifically, we start with hypercubes, which are known to have perfect state transfer in constant time, and describe how to construct a new family of graphs (with considerably less symmetry) with perfect state transfer at the same time.

\section{Isospectral Reductions}\label{sec:iso_red}
The isospectral reduction of a graph (square matrix) to a subset of its vertex set (to a principal sub-matrix) is a graph (matrix)-valued function that captures some of the original graph's (matrix's) spectral properties in a smaller graph (matrix).  We use notation from \cite{MR3237552,kempton2020characterizing}: let $\W$ be the set of rational function $p(\lambda)/q(\lambda)$ with complex coefficients in the variable $\lambda$ with no common factors and $deg(p)\leq deg(q)$.  Let $\W^{n\times n}$ be the set of $n\times n$ matrices whose entries come from $\W$.
\begin{definition}
 For a matrix $A=\mcdf$ the \emph{isospectral reduction} to the $M$ block (indexed by the set $S$) is given by \[\red=\rdef\in\W^{|S|\times|S|}.\]  
\end{definition}

Throughout this paper, $G$ will be a (possibly directed and/or weighted) graph (possibly with loops) on the vertex set $V$, $S$ a nonempty subset of $V$, and the adjacency matrix of $G$ will be $A=\bbm{M&C\\D&F}$, where the rows and columns of the $M$ block of $A$ correspond to the vertices in $S$. The isospectral reduction of a graph $G$ can be viewed as the weighted graph whose (weighted) adjacency matrix is the matrix isospectral reduction of the adjacency matrix of $G$.  As such, we will not hereafter distinguish between graph and matrix reductions, or indeed between graphs and matrices: $G$ will always be a graph, and $A$ the adjacency matrix of that graph.  See \cite{MR3237552,kempton2020characterizing} for detailed development of the theory of both graph and matrix isospectral reductions.  The conventional notation for the isospectral reduction $\rdef$ of $A$ to $S$ is $\mc{R}_S\ps{A}$, but in order to disambiguate which variable plays the role of $\lambda$ we will instead use the notation $\red$.
 The use of the formula $\rdef$ rather than the more conventional $A_{S\times S}-A_{S\times\overline{S}}\ps{A_{\overline{S}\times\overline{S}}-\li}^{-1}A_{\overline{S}\times S}$ is meant to improve readability and clarify the roles of each of the submatrices in the isospectral reduction. 

A remarkable fact that makes isospectral reductions useful and interesting is that they preserve the graph's eigenvalues in the following sense. 
\begin{theorem}[Theorem 2.2 of \cite{MR3237552}]\label{Eigenvalue preservation}
$A=\mcdf$ as above, 
\[\det\ps{\lambda I-\red}=\frac{\det\ps{\lambda I-A}}{\det\ps{\lambda I-F}}.\]
\end{theorem}

Isospectral reductions also satisfy the nice property that they can be applied in sequence to successively smaller subsets, and the result is independent of the order in which things are done, as seen in the following result.
\begin{theorem}[Theorem 2.5 of \cite{MR3237552}]\label{Reduction commutativity}
For any $S'\subseteq S$, $\mc{R}\ps{\lambda,S',\red}=\mc{R}\ps{\lambda,S',A}$.
\end{theorem}

Isospectral reductions can also capture some information about eigenvectors of a matrix \cite{duarte2015eigenvectors}.  Indeed, under some mild conditions on the set $S$, we have that the restriction of an eigenvector to $S$ remains an eigenvector after plugging in the associated eigenvalue into the isospectral reduction.

\begin{theorem}[Theorem 1 of \cite{duarte2015eigenvectors}]\label{thm:eigvect}
Let $Au=\lambda_0 u$.  Then plugging $\lambda_0$ into $\mathcal{R}(\lambda,S,A)$ satisfies
\[
\mathcal{R}(\lambda_0,S,A)u_S=\lambda_0u_S
\]
where $u_S$ denotes the restriction of $u$ to the set $S$.
\end{theorem}

\noindent Furthermore, the isospectral reduction $\red$ captures many walk counting properties of $A$ as they apply to $S$, and for instance is sufficient to tell whether two vertices in $S$ are cospectral in $G$.
\begin{theorem}[Theorem 3.3 and Corollary 3.4 of \cite{kempton2020characterizing}]
Two vertices $u$ and $v$ of a graph $G$ are cospectral if and only if $\mc{R}(\lambda,\{u\},A)=\mc{R}(\lambda,\{v\},A)$.  Equivalently, the $2\times2$ reduction $\mc{R}(\lambda,\{u,v\},A)$ has an automorphism.
\end{theorem} 



Much of the motivation for this paper is to find an ``unfolding" procedure that would let us construct graphs with interesting spectral properties by ``unfolding" isospectral reductions that exhibit the desired properties.  With this is mind, the remainder of this section will be dedicated to proving several preliminary technical results that will be useful later on.

\begin{lemma}\label{Edge pre-pre-splitting}
Let $A=\bbm{M_1+M_2&C_1&C_2\\D_1&F_1&0\\D_2&0&F_2}$ with the first block indexed by the set $S$. Then \[\red=\mc{R}\ps{\lambda,S,\bbm{M_1&C_1\\D_1&F_1}}+\mc{R}\ps{\lambda,S,\bbm{M_2&C_2\\D_2&F_2}}.\] 
\end{lemma}
\begin{proof}
We calculate the reduction:
\begin{align*}
\mc{R}\ps{\lambda,S,\bbm{M_1+M_2&C_1&C_2\\D_1&F_1&0\\D_2&0&F_2}}=&\ps{M_1+M_2}+\bbm{C_1&C_2}\ps{\li-\bbm{F_1&0\\0&F_2}}\inv\bbm{D_1\\D_2}\\
=&\ps{M_1+M_2}+\bbm{C_1&C_2}\bbm{\ps{\li-F_1}\inv&0\\0&\ps{\li-F_2}\inv}\bbm{D_1\\D_2}\\
=&M_1+M_2+C_1\ps{\li-F_1}\inv D_1+C_2\ps{\li-F_2}\inv D_2\\
=&\ps{M_1+C_1\ps{\li-F_1}\inv D_1}+\ps{M_2+C_2\ps{\li-F_2}\inv D_2}.
\end{align*}
\end{proof}

\begin{corollary}\label{Edge pre-splitting}
Let $A=\bbm{M_1+M_2+\ldots+M_k&C_1&C_2&\cdots&C_k\\D_1&F_1&0&\cdots&0\\D_2&0&F_2&\cdots&0\\\vdots&\vdots&\vdots&\ddots&\vdots\\D_k&0&0&\cdots&F_k}$ be a matrix with finitely many blocks and define $R_i=\mc{R}\ps{\lambda,S,\bbm{M_i&C_i\\d_i&F_i}}$. Then $\red=R_1+R_2+\ldots+R_k$.
\end{corollary}

\begin{theorem}\label{Correctness-preserving similarities}
If $X=\bbm{I&0\\0&Q}$ is partitioned conformal to $A=\mcdf$ with $Q$ invertible, then 
\[\mc{R}\ps{\lambda,S,XAX\inv}=\red.\]
\end{theorem}
\begin{proof}
First we calculate $\bbm{I&0\\0&Q}\inv A\bbm{I&0\\0&Q}$ more explicitly. 
\begin{align*}
\bbm{I&0\\0&Q}\inv A\bbm{I&0\\0&Q}=&\bbm{I&0\\0&Q}\inv \bbm{M&C\\D&F}\bbm{I&0\\0&Q}\\
=&\bbm{I&0\\0&Q\inv}\bbm{M&CQ\\D&FQ}\\
=&\bbm{M&CQ\\Q\inv D&Q\inv FQ}.
\end{align*}
Then we have
\begin{align*}
\mc{R}\ps{\lambda,S,\bbm{M&CQ\\Q\inv D&Q\inv FQ}}=&M+CQ\ps{\li-Q\inv FQ}\inv Q\inv D\\
=&M+CQ\ps{Q\inv\ps{\li-F}Q}\inv Q\inv D\\
=&M+CQQ\inv\ps{\li-F}\inv QQ\inv D\\
=&M+C\ps{\li-F}\inv D\\
=&\mc{R}\ps{\lambda,S,\bbm{M&C\\D&F}}\\
=&\red.
\end{align*}
\end{proof}

\begin{theorem}\label{Poles of the Reduction are eigenvalues of F}
The poles of $\mc{R}\ps{\lambda,S,\mcdf}$ (where $S$ indexes the $M$ block) occur at eigenvalues of $F$.
\end{theorem}
\begin{proof}
This follows directly from the definition $\red=\rdef$.
\end{proof}

\begin{theorem}\label{Constant part of the reduction is M}
Let $A=\mcdf$ where the $M$ block is indexed by $S$. The entries of the reduction $\red$ are rational functions in $\lambda$. Writing $\red$ in its canonical partial-fraction-decomposition form, the constant part of $\red$ is $M$.
\begin{proof}
The Frobenius norm of $\ps{\red-M}$ is equal to that of $C\lifi D$. This is bounded by $\Vs{C}\Vs{\lifi}\Vs{D}$. 
Since $\Vs{\lifi}$ can be made arbitrarily small by choosing large $\lambda$, the norm of this difference approaches $0$ as $\lambda\xrightarrow[]{}\infty$. Hence the limit of $\red$ as $\lambda\xrightarrow[]{}\infty$ is $M$.
\end{proof}
\end{theorem}

\begin{theorem}\label{Residues of the Reduction}
Let a pole $\mu$ of $\red$ be given where $F$ is a normal matrix. The residue of $\red$ at $\mu$ is $CE_{\mu}D$, where $E$ is the orthogonal projection onto the $\mu$-eigenspace of $F$.
\end{theorem}
\begin{proof}
Let the eigendecomposition of $F$ be $F=\sum\nu E_X$, with the sum ranging over eigenvalue-eigenspace pairs of $F$ and $E_X$ being the orthogonal projection matrix onto $X$. We may calculate
\begin{align*}
\lim_{\lambda\ar\nu}\ps{\lambda-\mu}\red=&\lim_{\lambda\ar\nu}\ps{\lambda-\mu}\ps{\rdef}\\
=&\lim_{\lambda\ar\nu}\ps{\lambda-\mu}\ps{M+C\ps{\lambda I-F}\inv D}\\
=&0+C\ps{\lim_{\lambda\ar\nu}\ps{\lambda-\mu}\ps{\lambda I-F}\inv}D\\
=&0+C\ps{\lim_{\lambda\ar\nu}\ps{\lambda-\mu}\ps{\lambda I-\sum\nu E_X}\inv}D\\
=&0+C\ps{\lim_{\lambda\ar\nu}\ps{\lambda-\mu}\sum\ps{\lambda-\nu}\inv E_X}D\\
=&0+C\ps{\lim_{\lambda\ar\nu}\ps{\lambda-\mu}\ps{\lambda-\mu}\inv E_{X_\mu}}D\\
=&CE_{X_{\mu}}D
\end{align*}
\end{proof}

\subsection{A New Perspective on Isospectral Reductions}\label{sec:new_form}

\begin{theorem}\label{Formula 2 for the reduction}
Given $A=\mcdf$ with rows and columns indexed by $V$ and partitioned as $S\cup\overline{S}$, then 
\[\red=\li-\ps{\Sigma^T\ps{\li-A}\inv\Sigma}\inv,\] where $\Sigma=\bbm{I\\0}$ has rows indexed by $S\cup\overline{S}$ and column indexed by $S$.
\end{theorem}
\begin{proof}
Examine the matrix $\bbm{0&I&0\\I&M&C\\0&D&F}$. Compute the reduction of this matrix to its first block in two different ways. First, reducing directly to the first block gives $0+\bbm{I&0}\ps{\li-\mcdf}\inv\bbm{I\\0}=\Sigma^T\ps{\li-A}\inv\Sigma$. Secondly, reduce in stages---first to the block defined by the first two partition classes as 
\begin{align*}
&\bbm{0&I\\I&M}+\bbm{0\\C}\ps{\li-F}\inv\bbm{0&D}\\=&\bbm{0&I\\I&M}+\bbm{0\\C\ps{\li-F}\inv}\bbm{0&D}\\=&\bbm{0&I\\I&M}+\bbm{0&0\\0&C\ps{\li-F}\inv D}\\=&\bbm{0&I\\I&M+C\ps{\li-F}\inv D}\\=&\bbm{0&I\\I&\red}
\end{align*}
and then to the first block directly: $0+I\ps{\li-\red}\inv I=\ps{\li-\red}\inv$. These two ways of calculating the reduction must give equivalent results by Theorem \ref{Reduction commutativity}, so $\ps{\li-\red}\inv=\Sigma^T\ps{\li-A}\inv\Sigma$. The result follows immediately.
\end{proof}

\subsection{Walk Generating Functions}\label{sec:gen_fun}
 While Theorem \cite{kempton2020characterizing} is interesting in its own right as an alternate way of thinking of and computing the isospectral reduction, we will see in this section how is reveals information about walk generating functions.  This stems from a combinatorial interpretation of the walk generating functions.

We will recall some notation from $\cite{kempton2020characterizing}$.  Let $G$ be a graph and $S$ a subset of the vertex set.  We define two walk generating functions.  First, define \[W_S(t) = \sum_{\ell=0}^\infty w_\ell(S)t^\ell\] where $w_\ell(S)$ denotes the $|S|\times |S|$ matrix whose $(a,b)$ entry is the number of walks of length $\ell$ in $G$ beginning at $a$ and ending at $b$. That is, $W_S(t)$ is the generating function for enumerating walks in the subset $S$.  It is easy to see that \[W_S(t) = \left((I-tA)^{-1}\right)_{S,S}.\]  We will denote by \[W_S^*(t) = \sum_{\ell=1}^\infty w_\ell^*(S)t^\ell\] where $w_\ell^*(S)$ is the $|S|\times|S|$ matrix whose $(a,b)$ entry is the number of walks of length $\ell$ from $a$ to $b$ in $G$ that leave $S$ on their first step and do not return till the last step (the $S$-non-returning walks from $a$ to $b$).  In \cite{kempton2020characterizing}, the following relationship between $W_S^*(t)$ and $\mc{R}(\lambda,S,A)$ was obtained.

\begin{lemma}[Theorem 3.6 of \cite{kempton2020characterizing}]\label{lem:redgenfunct}
\[\red=\frac1\lambda W^*_S\left(\frac{1}{\lambda}\right).\]
\end{lemma}

Furthermore, in \cite{kempton2020characterizing}, it was proven that in the specific instance when $S=\{a\}$ consists of a single vertex, then the relationship between $W_a$ and $W_a^*$ is given by
\[
W_a(t)=\frac{1}{1-W_a^*(t)}.
\]

This can be proven from direct combinatorial methods, or from identities involving isospectral reductions (see Section 3.2 of \cite{kempton2020characterizing}). Theorem \ref{Formula 2 for the reduction} immediately gives a generalization of this identity to arbitrary subsets.

\begin{theorem}\label{thm:gen_funct}
Let $G$ be a graph and $S$ some subset of the vertex set.  Then
\[
W_S(t) = (I-W_S^*(t))^{-1}.
\]
\end{theorem}
\begin{proof}
Let $t=\frac1\lambda$.  From Theorem \ref{Formula 2 for the reduction} we have
\begin{align*}
    \red &= \lambda I - \left(\left((\lambda I-A)^{-1}\right)_{S,S}\right)^{-1}\\
    &=\frac1t I - \frac1t\left(\left(( I-tA)^{-1}\right)_{S,S}\right)^{-1}.
\end{align*}
Then from Lemma \ref{lem:redgenfunct} we have that
\[
W^*_S(t) = I-W_S(t)^{-1}
\]
from which the result follows.
\end{proof}

Thus, any property of a subset of a graph that can be determined by the walk generating function can also be determined by the isospectral reduction to that subset.

We may interpret the formula $\ps{\li-\red}\inv=\Sigma^*\ps{\li-A}\inv\Sigma$ to say that the $S$-non-returning walk generating function provided by the isospectral reduction of $A$ to $S$ is \textit{exactly what we need to use in place of edges} to calculate the walk generating function of $A$ restricted to $S$ using the normal formula. Isospectral reductions are like generalized edges in this sense. Just as edges are the minimal walks in a graph which compose to give all other walks and walk-count functions through the walk generating function power series, non-returning walks are the minimal walks from the set $S$ to itself which compose by the same power series to give all walks from $S$ to itself. 


\subsection{Generalized Isospectral Reductions}\label{sec:gen_red}
In addition to giving us another way to compute isospectral reductions, the alternate reduction formula from Theorem \ref{Formula 2 for the reduction} suggests a generalization of the isospectral reduction. 
\begin{definition}\label{Generalized reduction definition}
For a matrix $\Sigma$ with full column rank and rows indexed by the vertex set of the graph with adjacency matrix $A$, define the generalized reduction of $A$ with respect to $\Sigma$ in terms of $\lambda$ to be \[\mc{R}\ps{\lambda,\Sigma,A}:=\li-\ps{\Sigma^*\ps{\li-A}\inv\Sigma}\inv.\]
\end{definition}

\begin{remark}
    In our definition, we are requiring $\Sigma$ to have orthonormal columns.  We could define this more generally where $\Sigma$ simply has linearly independent columns, and replace $\Sigma^*$ with the pseudoinverse of $\Sigma$ in the definition.  However, the assumption of orthonormal column will make some of the theory easier, and will be sufficiently general for all that we will be doing.
\end{remark}

This formula provides a new way of thinking of many of the classic results about isospectral reductions.  For instance we have the following.

\begin{theorem}
Let $\Sigma_1$ be an $n\times k$ matrix with orthonormal columns and $\Sigma_2$ a $k\times r$ matrix with orthonormal columns.  Then $\Sigma_1\Sigma_2$ has orthonormal columns and
\[\mc{R}(\lambda,\Sigma_1\Sigma_2,A) = \mc{R}(\lambda,\Sigma_2,\mc{R}(\lambda,\Sigma_1,A)).\]
\end{theorem}
The proof is a straightforward computation, and Theorem \ref{Reduction commutativity} can be viewed as a simple special case.

\begin{theorem}
Let $\Sigma$ have orthonormal columns, and let $\Delta$ be a matrix whose columns are an orthonormal basis for the orthogonal complement of the columns space of $\Sigma$.
We have \[\det(\lambda I-\mc{R}(\lambda,\Sigma,A)) = \frac{\det(\lambda I-A)}{\det(\lambda I-\Delta^* A\Delta)}.\]
\end{theorem}
\begin{proof}
Let $X$ be the matrix with the same number of rows and columns as $\Sigma$ whose $i\Th$ column is the $i\Th$ basis vector and let $Y$ be the matrix such that $\bbm{X&Y}=I$ and let $U=\bbm{\Sigma&\Delta}$. Note that $\mc{R}\ps{\lambda,X,A}$ is the standard reduction of $A$ to its first $\dim\ps{\col\ps{\Sigma}}$ rows and columns. Note also that $\Sigma=UX$ and $\Sigma^*=X^*U^*$. We have
\begin{align*}
\det({\li-\mc{R}\ps{\lambda,\Sigma,A}})=&\det\left({\li-\ps{\li-\ps{\Sigma^*\ps{\li-A}\inv\Sigma}\inv}}\right)\\
=&\det\left({\li-\ps{\li-\ps{X^*U^*\ps{\li-A}\inv UX}\inv}}\right)\\
=&\det\left({\li-\ps{\li-\ps{X^*\ps{\li-U^*AU}\inv X}\inv}}\right)\\
=&\det\left({\li-\mc{R}\ps{\lambda,X,U^*AU}}\right)\\
=&\frac{\det\left({\li-A}\right)}{\det\left({\li-Y^*U^*AUY}\right)}\\
=&\frac{\det\left({\li-A}\right)}{\det\left({\li-\Delta^*A\Delta}\right)}
\end{align*}
by Theorem \ref{Eigenvalue preservation}.
\end{proof}

Finally, we can generalize Theorem \ref{thm:eigvect} to apply to eigenvectors of generalized reductions.
\begin{theorem}
If $Au=\lambda_0u$, then \[\mathcal{R}(\lambda_0,\Sigma,A)\Sigma^*u = \lambda_0\Sigma^*u.\]
\end{theorem}
\begin{proof}
Writing $\Sigma = UX$ where $X^*=[I~0]$, the result follows directly from Theorem \ref{thm:eigvect} applied to $U^*AU$, similar to the above.
\end{proof}

\subsection{Divisor Matrices of Equitable Partitions}\label{sec:equit_part}

In this section, we will see that the theory of divisor matrices of equitable partitions can be viewed as a special case of the generalized isospectral reduction.  Recall that a partition $\Pi=(V_1,...,V_k)$ of the vertex set of a graph $G$ is called an \emph{equitable partition} of $G$ if, for all $i,j$ (including $i=j$), there are constants $c_{i,j}$ such that any vertex from part $V_i$ of the partition has exactly $c_{i,j}$ neighbors in part $V_j$.  Equivalently, if we look at the block partition of the adjacency matrix $A$ of $G$, each block has constant row sums.

An equivalent way of looking at equitable partitions is as follows.  Given a partition $\Pi$ with $k$ parts of $G$, let $P$ be the $n\times k$ \emph{indicator matrix} of $\Pi$---that is, the matrix whose columns are indicator vectors for the parts of the partition.  Then $\Pi$ is an equitable partition of $G$ if and only if there is some $k\times k$ matrix $d$ with $AP=Pd$.  Here $d$ is called the \emph{divisor matrix} or \emph{quotient matrix} of the equitable partition.  It is well known that eigenvalues of $d$ are eigenvalues of $A$, and corresponding eigenvectors $x$ of $d$ lift to eigenvectors $Px$ of $A$. See, for example, \cite[Chapter 9.3]{godsil2001algebraic} for a detailed development of the theory of equitable partitions. 

It is common in studying equitable partitions to replace the indicator matrix $P$ with the corresponding matrix whose columns have been normalized.  Then we still have $AP=Pd$ for some $k\times k$ matrix $d$, but now since $P^TP=I$, we can directly express $d=P^TAP$, which is a symmetric matrix.  We will refer to this as the \emph{symmetrized divisor matrix}.

\begin{example}
For instance, suppose $G$ is the following graph .
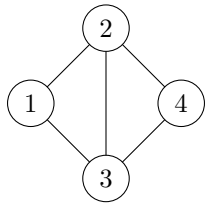
\begin{center}
    \begin{tikzpicture}
    \draw (-1,0)node[circle,draw, fill=white]{1} -- (0,1)node[circle,draw, fill=white]{2} -- (0,-1)node[circle,draw, fill=white]{3} -- (1,0)node[circle,draw, fill=white]{4} -- (0,1) (-1,0)--(0,-1);
    \end{tikzpicture}
\end{center}
It is easy to see $G$ has an equitable partition $(\{1\},\{2,3\},\{4\})$.  The indicator matrix and divisor matrix, respectively, are \[P=\bbm{1&0&0\\0&1&0\\0&1&0\\0&0&1}, d=\bbm{0&2&0\\1&1&1\\0&2&0}.\] Normalizing the columns gives normalized indicator matrix and symmetrized divosor matrix
\[
P=\bbm{1&0&0\\0&\frac{1}{\sqrt{2}}&0\\0&\frac{1}{\sqrt{2}}&0\\0&0&1}, d=\bbm{0&\sqrt{2}&0\\\sqrt{2}&1&\sqrt{2}\\0&\sqrt{2}&0}.
\]
Both versions of $d$ have eigenvalues $0, (1\pm\sqrt{17})/2$ which are also eigenvalues of $A$.
\end{example}

\begin{theorem}\label{thm:equit}
Let $G$ be a graph with adjacency matrix $A$, and suppose $G$ has an equitable partition with normalized indicator matrix $P$ and symmetrized divisor matrix $d$.  Then
\[
d=\mc{R}(\lambda,P,A).
\]
\end{theorem}
\begin{proof}
Starting from the equation $AP=Pd$, we calculate
\begin{align*}
AP=&Pd\\\lambda P-AP=&\lambda P-Pd\\\ps{\li-A}P=&P\ps{\li-d}\\P\ps{\li-d}\inv=&\ps{\li-A}\inv P\\P^TP\ps{\li-d}\inv=&P^T\ps{\li-A}\inv P\\\ps{\li-d}\inv=&P^T\ps{\li-A}\inv P\\d=&\li-\ps{P^T\ps{\li-A}\inv P}\inv\\d=&\mc{R}\ps{\lambda,P,A}.
\end{align*}
\end{proof}
Thus the normalized divisor matrices of equitably partitioned matrices are exactly the same as the generalized reductions of those matrices with respect to the normalized partition matrices of those equitable partitions. 
Note that it is remarkable that in the case of the normalized partition matrix $P$, when there is an equitable partition, this generalized reduction works out in such a way that all dependence on the variable $\lambda$ cancels and the reduction leaves us with a matrix with coefficients from the base field.  This ultimately stems from the equation $AP=Pd$, and indeed, whenever $\Sigma$ has a column space that is invariant under the action of $A$, we will see a similar cancellation of the dependence on $\lambda$ in the generalized isospectral reduction. Indeed, if we took $\Sigma$ to consist of orthonormal eigenvectors of $A$, then the generalized reduction using $\Sigma$ would simply yield a diagonal matrix with the corresponding eigenvalues on the diagonal.  Thus the theory of diagonalization of a Hermitian matrix can be viewed a specific case of generalized isospectral reductions as well.

\section{Unfolding Isospectral Reductions}\label{sec:unfolding}

This section will be dedicated to investigating the problem of constructing a matrix with entries in $\R$ or $\C$ that has a given isospectral reduction.  We will refer to this process as ``unfolding" the isospectral reduction.

\begin{definition}
A matrix $R\in \W^{s\times s}$ is \emph{unfoldable} if there exists some matrix $A\in \C^{n\times n}$ (for some $n>s$) and some subset $S$ of the index set of $A$ with $|S|=s$ such that \[\mc{R}(\lambda,S,A)=R.\]  We will call such a matrix $A$ an \emph{unfolding} of $R$.
\end{definition}


\begin{theorem}\label{Edge Splitting}
If the (finitely many) matrices $R_1,R_2,\dots,R_k$ are unfoldable, then their sum is also unfoldable.
\end{theorem}
\begin{proof}
This is a straightforward corollary of Corollary \ref{Edge pre-splitting}. The $M$ blocks of the unfoldings are added together, the $C$ blocks are added to the right, the $D$ blocks are all present going downwards, and the $F$ blocks grow diagonally downwards.
\end{proof}

\begin{lemma}\label{basic unfoldings}
Any matrix of the form $R\ps{\lambda}=\frac{1}{\ps{\lambda-\nu}^n}K$, where $K\in \C^{s\times s}$ and $n\in\N$, is unfoldable. Furthermore, in the case where $n=1$, $\nu$ is real, and $K$ is positive semidefinite, $R\ps{\lambda}$ has a Hermitian unfolding, and these conditions are necessary for such an $R$ to have a Hermitian unfolding.
\end{lemma}
\begin{proof}
For the non-Hermitian case we work by induction with the base case $n=1$
. 
choose any matrices $X$ and $Y$ such that $XY=K$ (such as the rank decomposition of $K$). Then the matrix $\bbm{0&X\\Y&\nu I}$ has the reduction (to its first block) $0+X\ps{\li-\nu I}\inv Y=\ps{\lambda-\nu}\inv XY=R\ps{\lambda}$, so unfoldings exist in all cases for $n=1$ by construction. 

Note that for $K$ positive semidefinite and $\nu$ real we could choose $X,Y$ such that $Y=X^*$, and in this case $\bbm{0&X\\Y&\nu I}$ becomes a Hermitian unfolding of $R$, proving the sufficiency part of our result for Hermitian matrices.

Continuing with the induction, assume that unfoldings always exist for matrices of the specified form for a given $n\geq1$. Consider an arbitrary matrix $R\ps{\lambda}=\frac{1}{\ps{\lambda-\nu}^{n+1}}K$ of the ``next" form. Once again, write $K=XY$. 
By our inductive assumption, $\frac{1}{\ps{\lambda-\nu}^n}X$ has some unfolding- call it $\mcdf$. Since $\frac{1}{\ps{\lambda-\nu}^n}X$ approaches $0$ as $\lambda\xrightarrow[]{}\infty$, we have from theorem \ref{Constant part of the reduction is M} that $M=0$.
Finally, consider the matrix $\bbm{0&0&C\\Y&\nu I&0\\0&D&F}$. Reducing this matrix to its first block yields
\begin{align*}
&0+\bbm{0&C}\ps{\bbm{\li&0\\0&\li}-\bbm{\nu I&0\\D&F}}\inv\bbm{Y\\0}\\
=&\bbm{0&C}\ps{\bbm{\ps{\lambda-\nu}I&0\\-D&\li-F}}\inv\bbm{Y\\0}\\
=&\bbm{0&C}\bbm{\ps{\lambda-\nu}\inv I&0\\\ps{\li-F}\inv D\ps{\lambda-\nu}\inv&\ps{\li-F}\inv}\bbm{Y\\0}\\
=&C\ps{\li-F}\inv D\ps{\lambda-\nu}\inv Y\\
=&\ps{M+C\ps{\li-F}\inv D}\ps{\lambda-\nu}\inv Y\\
=&\frac{1}{\ps{\lambda-\nu}^n}X\ps{\lambda-\nu}\inv Y\quad\text{By construction of }M,C,D,F\\
=&\frac{1}{\ps{\lambda-\nu}^n}\ps{\lambda-\nu}\inv XY\\
=&\frac{1}{\ps{\lambda-\nu}^{n+1}}K
\end{align*}
so $\frac{1}{\ps{\lambda-\nu}^{n+1}}K$ indeed has an unfolding. Since $K$ and $\nu$ were arbitrary, the existence of arbitrary unfoldings for matrices of the form $\frac{1}{\ps{\lambda-\nu}^n}K$ implies the same for matrices of the form $\frac{1}{\ps{\lambda-\nu}^{n+1}}K$, and by induction on the base case $n=1$ we have the first part of our theorem. 

For necessity of the Hermitian unfoldability conditions, consider the Hermitian matrix $\mcdf$ and its reduction $\rdef$ to its first block. Now $F$ must be Hermitian, so it has the eigendecomposition $F=\sum\nu_i\proj_{V_i}$ expressing it as a real linear combination of projections onto its eigenspaces (the sum is taken over eigenvalue-eigenspace pairs of $F$). Then $\li-F=\sum\ps{\lambda-\nu_i}\proj_{V_i}$ and $\lifi=\sum\ps{\lambda-\nu_i}\inv\proj_{V_i}$. Finally, the reduction of $\mcdf$ to its first block is $M+\sum\ps{\lambda-\nu_i}\inv C\proj_{V_i}C^*$. Write $C\proj_{V_i}C^*=K_i$ and cancel off zero terms to get $R\ps{\lambda}=M+\sum_{i\in\mc{I}}\ps{\lambda-\nu_i}\inv K_i$, where $\mc{I}$ is some indexing set. Since the nonzero terms are canceled and each $\nu_i$ is distinct, $R\ps{\lambda}$ has exactly $\vs{\mc{I}}$ poles. There must only be a single term left\footnote{Or none, in which case $R$ is the constant zero matrix, trivially has itself as a Hermitian unfolding, and trivially satisfies the conditions in the conclusion with $K=0$, $\nu=0$, $n=1$} (and $M$ must equal $0$) if $R\ps{\lambda}$ is to take the form $\frac{1}{\ps{\lambda-\nu}^n}K$, since the latter only ever has one pole (and has zero constant part). Hence any $R\ps{\lambda}$ arising as the reduction of a Hermitian matrix and having the desired form takes the form $\frac{1}{\lambda-\nu_i}K_i$ for some real $\nu_i$ (an eigenvalue of the Hermitian $F$-block of its Hermitian unfolding). Furthermore, since $K_i=C\proj_{V_i}C^*$ is star-congruent to a positive semidefinite projection matrix, it is itself positive semidefinite. Then the conditions $K$ is positive semidefinite, $\nu$ is real, $n=1$ are both necessary and sufficient for $\frac{1}{\ps{\lambda-\nu}^n}K$ to have a Hermitian unfolding, as desired.
\end{proof}

\begin{theorem}\label{Unfolding Existence and Procedure}
Every matrix-valued function $R\ps{\lambda}$ of $\lambda$ having as its entries rational functions in $\lambda$ with numerator degrees not exceeding their respective denominators' degrees has an unfolding. Furthermore, such a matrix has a Hermitian unfolding iff all of its (entries') poles are real and simple, the residue matrix at each of those those poles (calculated such that the $i,j$ entry of the $\nu$-residue matrix is the residue of the function $R_{i,j}\ps{\lambda}$ at $\nu$) is positive semidefinite, and $\lim_{\lambda\xrightarrow[]{}\infty}R\ps{\lambda}$ is Hermitian.
\end{theorem}
\begin{proof}
Let $\mc{I}$ index the poles of $R\ps{\lambda}$ (a number is a pole of $R\ps{\lambda}$ if it is the pole of some entry of $R\ps{\lambda}$), label the associated poles $\nu_i$, and label their orders $n_i$. Then we can take a partial fraction decomposition (entrywise) of $R\ps{\lambda}$ to get $R\ps{\lambda}=M+\sum_{i\in\mc{I}}\sum_{k=1}^{n_i}\frac{1}{\ps{\lambda-\nu_i}^k}K_{i,k}$, where there are no positive powers of $\lambda$ in the sum because of the degree conditions. $M$ is unfoldable because it is a constant matrix, and hence reduces to itself. Every other term in the sum is unfoldable by Lemma \ref{basic unfoldings}. By Theorem \ref{Edge Splitting} we have the result. Furthermore, both theorems give algorithms for constructing their associated unfoldings, so combining them with the partial fraction decomposition in this way gives us a general unfolding procedure. First partial-fraction-decompose, then find $M$ as the constant term, then find unfoldings for each term by rank-decomposing the coefficient matrices, then combine all the unfoldings by appending $C$ blocks to the right of $M$, $F$ blocks on the diagonal, and $D$ blocks downwards from $M$.

The case for Hermitian unfoldings is similar. Note that if the unfoldings $\bbm{M_1&C_1\\D_1&F_1}$ and $\bbm{M_2&C_2\\D_2&F_2}$ are Hermitian then $\bbm{M_1+M_2&C_1&C_2\\D_1&F_1&0\\D_2&0&F_2}$ will also be; this enables us to still use Theorem \ref{Edge Splitting} to construct the unfoldings. If all of our stated conditions are met then we can use Lemma \ref{basic unfoldings} to unfold each term in a Hermitian way (the trivial unfolding of the $M$ block will clearly be Hermitian as well by the last condition); it remains to prove necessity.

The condition that $\lim_{\lambda\xrightarrow[]{}\infty}R\ps{\lambda}$ be Hermitian is clearly necessary, since this limit is equal to the $M$ block of any unfolding by Theorem \ref{Constant part of the reduction is M}, and is hence a principal submatrix of every unfolding. Principal submatrices of Hermitian matricies must always be Hermitian. For the necessity of the other conditions we refer back to the proof of Lemma \ref{basic unfoldings}, where we showed that the reduction of a Hermitian matrix takes the form $M+\sum_{i\in\mc{I}}\frac{1}{\lambda-\nu_i}K_i$ where each $\nu_i$ is the eigenvalue of a Hermitian matrix (and hence real) and each $K_i$ is star-congruent to an orthogonal projection matrix (and hence positive semidefinite). Anything of this form clearly has simple real poles and positive semidefinite residues at those poles, so the remaining conditions are also necessary as well as sufficient.
\end{proof}The unfolding procedure can now produce general weighted digraphs from general $\mathbb{W}$-matrices and graphs with Hermitian adjacency matrices from $\mathbb{W}$-matrices satisfying a few additional conditions. It is natural to look for ways to continue specializing the procedure and eventually to seek out an unfolding procedure for unweighted symmetric zero-one graphs without loops. Let us start by removing loops from the graph produced by the unfolding procedure.

\begin{theorem}\label{Loopless unfolding}
Any reduction $R$ with a Hermitian unfolding such that $\lim_{\lambda\ar\infty}R\ps{\lambda}$ has only zeros on its diagonal (is ``hollow") has in particular a hollow Hermitian unfolding. 
\end{theorem}
\begin{proof}
Let $A=\mcdf$ be the unfolding of $R$ obtained by the Hermitian unfolding procedure above. The condition that $\lim_{\lambda\ar\infty}R\ps{\lambda}$ is hollow is exactly the condition that $M$ is hollow (so said condition is necessary for any matrix to have a hollow unfolding). Recall \ref{Correctness-preserving similarities}. The strategy of our proof will be to construct a correctness-preserving similarity sending a general Hermitian unfolding of $R$ to a hollow one. First we have to fix $A$'s trace. Note that $\bbm{M&C&0\\C^*&F&0\\0&0&-\tr\ps{F}}$ is also a Hermitian unfolding of $R$. We seek to show that the Hermitian matrix $\bbm{F&0\\0&-\tr\ps{F}}$ is unitarily similar to a hollow matrix, say by $U^*\bbm{F&0\\0&-\tr\ps{F}}U=F'$. Then by correctness-preserving similarities \ref{Correctness-preserving similarities} we will have that $\bbm{M&\bbm{C&0}U\\U^*\bbm{C\\0}&F'}$ is also a Hermitian unfolding for $R$, and in particular a hollow one. We only need the following lemma:
\begin{lemma}
Any Hermitian matrix with trace zero is unitarily similar to a hollow matrix.
\end{lemma}
\begin{proof}
We proceed by induction on the size of the matrix. Clearly the result holds for one-by-one matrices, as we only need to check that $\bbm{0}$ is hollow. Suppose that it holds for $n\times n$ matrices. Then let a general $n+1\times n+1$ trace-zero Hermitian matrix $A$ be given. Because $A$ has zero trace, it either has all zero eigenvalues (and then is the zero matrix by Hermiticity and satisfies our desired conclusion without the application of any unitary similarities) or it has some positive and some negative eigenvalues with corresponding eigenvectors. Consider the function sending an $n+1$-dimensional vector $x$ to $x^*Ax$. This function is obviously continuously differentiable on the $n+1$-dimensional unit sphere. The unit sphere is also simply connected and contains vectors $x_+,x_-$ for which the function is positive and negative, respectively (take them to be unit eigenvectors of $A$ corresponding to positive and negative eigenvalues). Then by the intermediate value theorem there is some unit vector $x_0$ for which $x_0^*Ax_0=0$. Pick an orthonormal basis for the orthogonal complement of $x_0$'s span and make the vectors comprising this basis the columns of a new matrix $V$. Now $U=\bbm{x_0&V}$ is unitary because it has orthonormal columns, and $U^*AU=\bbm{x_0^*Ax_0&x_0^*AV\\V^*Ax_0&V^*AV}=\bbm{0&x_0^*AV\\V^*Ax_0&V^*AV}$. Because unitary similarities preserve the trace and Hermiticity, $V^*AV$ is a trace-zero $n\times n$ Hermitian matrix. By our inductive hypothesis there is a unitary matrix $W$ such that $W^*V^*AVW$ is hollow. Let $X=U\bbm{1&0\\0&W}$. Then $X^*AX=\bbm{0&x_0^*AVW\\W^*V^*Ax_0&W^*V^*AVW}$ is a hollow Hermitian matrix which is unitarily similar to $A$. By induction our lemma is proved.
\end{proof}
Using the lemma we construct a unitary matrix sending $\bbm{F&0\\0&-\tr\ps{F}}$, turn that into a correctness-preserving unitary similarity, and use that similarity to give us a new Hollow Hermitian unfolding for $R$ as discussed above.
\end{proof}

The reduction tells us the restrictions of all powers of $A$ to the $M$ block, and hence the eigenvectors of $A$ as projected to the $S$-subspace. If all of these vectors can be chosen to be purely real then the Hermitian unfolding procedure can be chosen as a real-symmetric unfolding procedure and the hollowing procedure above can be performed to give a real-weighted symmetric loopless unfolding procedure. There is great difficulty in general in making the entries of such graphs integers by means of correctness-preserving transformations. If that can be accomplished the result can often be treated as a residue matrix to give an unfolding with entries in $\cs{-1,0,1}$. If they can be made positive integers it is possible to obtain a simple graph unfolding.

\section{Quantum Walks}\label{sec:quantum_walk}

As mentioned in the introduction, a quantum walk on a graph $G$ is described by the unitary transition matrix
\[
U(t)=e^{-itA}
\]
where $A$ is the adjacency matrix of $G$. We say there is perfect state transfer (PST) from $u$ to $v$ at time $\tau$ if \[U(\tau) e_u = \gamma e_v\] for some unit complex number $\gamma$.  Since $U(t)$ is a unitary matrix, PST at time $\tau$ is equivalent to saying that $U(\tau)$ is block diagonal, with a block indexed by $u$ and $v$ equal to $\bbm{0&\gamma\\\gamma&0}$.

A generalization of PST called \emph{fractional revival} is defined as follows.  We say fractional revival (FR) occurs on a subset $S$ of the vertices of $G$ at time $\tau$ if, for any $\phi$ supported only on $S$, $U(\tau)\phi$ is also supported only on $S$.  Fractional revival at time $\tau$ is equivalent to the existence of an $|S|\times |S|$ unitary matrix $H$ such that \[U(\tau)=\bbm{H&0\\0&U'}\] for some $U'$.  We call this fractional revival on $S$ with respect to $H$.  See \cite{chan2022fundamentals,chan2019quantum} for more details.  Note that PST from $u$ to $v$ is FR with respect to the matrix $\bbm{0&\gamma\\\gamma&0}$ on the set $\{u,v\}$.

Both PST and FR have approximations, called, respectively, pretty good state transfer (PGST) and pretty good fractional revival (PGFR).  PGFR amounts to saying that for every $\epsilon>0$ there is a time $\tau$ for which $U(\tau)$ has a block decomposition $\bbm{H&M\\M^*&U'}$ with $||M||<\epsilon$, and PGST is a similar approximation to PST.  See \cite{chan2021pretty,godsil2012number} for more details.  

The goal of this section will be to study the quantum walk matrix, and hence each of these quantum communication phenomena, by way of the isospectral reduction.  We will use some well-known tools related to the Laplace transform of a function.  Recall that the Laplace transform $\mathscr{L}_{t\ar s}$ is a linear operator from functions in $t$ to functions in $s$ defined by \[\mathscr{L}_{t\ar s}f\ps{t}=\int_{0}^{\infty}f\ps{t}e^{-st}dt,\] and $\mathscr{L}_{t\ar s}\inv$ denotes its inverse. The Laplace transform is defined to act element-wise on matrix- or vector-valued functions (and hence commutes with all left- and right- matrix actions on matrix-valued functions by linearity). We will in particular use the fact that \[\mathscr{L}_{t\ar s}e^{At}=\ps{sI-A}\inv\] for any matrix $A$.  See any standard text (e.g. \cite{doetsch2012introduction}) for details on the Laplace transform.

Our main result of this section establishes an equivalence between the information provided by the restricted quantum walk matrix $\Sigma^*U\ps{t,A}\Sigma$ and the corresponding generalized isospectral reduction $\mc{R}\ps{\lambda,\Sigma,A}$.

\begin{theorem}\label{Quantum walk equivalence to reduction}
Let $U\ps{t,A}:=e^{-itA}$ be the quantum walk matrix on $A$ in terms of the variable $t$. Then \[\Sigma^*U\ps{t,A}\Sigma=\mathscr{L}\inv_{t\xrightarrow[]{}s}\ps{\ps{sI+i\mc{R}\ps{is,\Sigma,A}}\inv}\] and \[\mc{R}\ps{\lambda,\Sigma,A}=\lambda I-i\ps{\mathscr{L}_{t\rightarrow -i\lambda}\ps{\Sigma^*U(t,A)\Sigma}}^{-1}.\] 
\end{theorem}
\begin{proof}
We compute as follows:
\begin{align*}
\Sigma^*U\ps{t,A}\Sigma=&\Sigma^*e^{-itA}\Sigma\\
=&\Sigma^*\ps{\mathscr{L}\inv_{t\xrightarrow[]{}s}\ps{sI+iA}\inv}\Sigma\\
=&\mathscr{L}\inv_{t\xrightarrow[]{}s}\Sigma^*\ps{i\ps{isI-A}\inv}\Sigma\\
=&i\mathscr{L}\inv_{t\xrightarrow[]{}-i\lambda}\Sigma^*\ps{\li-A}\inv\Sigma\\
=&i\mathscr{L}\inv_{t\xrightarrow[]{}-i\lambda}\ps{\li-\ps{\li-\ps{\Sigma^*\ps{\li-A}\inv\Sigma}\inv}}\inv\\
=&i\mathscr{L}\inv_{t\xrightarrow[]{}-i\lambda}\ps{\li-\mc{R}\ps{\lambda,\Sigma,A}}\inv\\
=&\mathscr{L}\inv_{t\xrightarrow[]{}-i\lambda}\ps{-i\li+i\mc{R}\ps{\lambda,\Sigma,A}}\inv\\
=&\mathscr{L}\inv_{t\xrightarrow[]{}s}\ps{Is+i\mc{R}\ps{is,\Sigma,A}}\inv.
\end{align*}

The second formula is a direct corollary to the first. 
\end{proof}

Note that this is the second theorem of this form. Earlier we saw that the walk-generating function for a graph (as restricted to a subset $S$ of the graph's vertex set) can be obtained by the similar-looking formula $\ps{\sum_{i=0}^\infty A^it^i}_{S\times S}=\ps{I-t\mc{R}\ps{\frac{1}{t},S,A}}\inv$. In fact, the restriction of any power series in $A$ to a subset can be obtained from the isospectral reduction of $A$ to that subset by passing to the walk-generating function and taking its Maclaurin expansion to get each term $\ps{A^i}_{S\times S}$ as $i$ ranges through the natural numbers. If the power series has no zero coefficients then a reverse formula also exists. 

\begin{corollary}
If $G$ and $H$ are two graphs, and $S\subset V(G)$ and $T\subset V(H)$ are such that $\mc{R}(\lambda,S,A_G) = \mc{R}(\lambda, T,A_H)$, then there is perfect state transfer between two vertices of $S$ in $G$ if and only if there is perfect state transfer between the two corresponding vertices of $T$ in $H$.
\end{corollary}

We could state similar corollaries for any other local quantum state transfer phenomenon, such as pretty good state transfer, fraction revival, pretty good fractional revival, etc.

\section{Examples}\label{sec:examples}

\subsection{Examples from Hypercubes}

Consider the hypercube graph $Q_n$. Because $Q_n$ is the Cartesian product of $n$ $2$-paths, and $2$-paths exhibit perfect state transfer between their endpoints at time $t=\frac\pi2$, $Q_n$ also exhibits perfect state transfer between any pair of antipodal vertices at $t=\frac\pi2$ (see \cite{godsil_survey} for more information). Choose a pair $a,b$ of antipodal vertices. The distance partition gives an equitable partition of $Q_n$ which places $a,b$ into singlet classes. 
This is associated to a matrix $P$ with orthonormal columns such that $AP=Pd$ for some matrix $d$, (where $A$ is the adjacency matrix of $Q_n$), and such that the columns corresponding to the $\cs{a},\cs{b}$ classes of the distance partition are the standard basis vectors $e_a,e_b$ corresponding to $a,b$. 
Let $A'$ be any equitably partitionable matrix with the same divisor matrix as $A\ps{Q_n}$, $A'P'=P'd$. By Theorems \ref{Reduction commutativity} and \ref{thm:equit}, 
\begin{align*}
\mc{R}\ps{\lambda,\cs{a,b},A}=&\mc{R}\ps{\lambda,\bbm{e_a&e_b},A}\\=&\mc{R}\ps{\lambda,P\bbm{e_{\cs{a}}&e_{\cs{b}}},A}\\=&\mc{R}\ps{\lambda,\bbm{e_{\cs{a}}&e_{\cs{b}}},\mc{R}\ps{\lambda,P,A\ps{Q_n}}}\\=&\mc{R}\ps{\lambda,\bbm{e_{\cs{a}}&e_{\cs{b}}},d}\\=&\mc{R}\ps{\lambda,\bbm{e_{\cs{a}}&e_{\cs{b}}},\mc{R}\ps{\lambda,P',A'}}\\=&\mc{R}\ps{\lambda,P'\bbm{e_{\cs{a}}&e_{\cs{b}}},A'}\\=&\mc{R}\ps{\lambda,\bbm{e_a&e_b},A'}\quad a,b\text{ are singlets of the partition}\\=&\mc{R}\ps{\lambda,\cs{a,b},A'}
\end{align*}
and so by Theorem \ref{Quantum walk equivalence to reduction}, $\bbm{U\ps{t,A}}_{\cs{a,b}\times\cs{a,b}}=\bbm{U\ps{t,A'}}_{\cs{a,b}\times\cs{a,b}}$ and so $A'$ also exhibits perfect state transfer at time $t$ between vertices $a$ and $b$. For instance, the equitably partitioned matrix
\[\left[\begin{array}{c|cccc|cccccc|cccc|c}
    0 & 1&1&1&1&0&0&0&0&0&0&0&0&0&0&0 \\
    \hline
    1 &0&0&0&0&1&1&1&0&0&0&0&0&0&0&0 \\
    1 &0&0&0&0&0&0&0&1&1&1&0&0&0&0&0 \\
    1 &0&0&0&0&1&1&0&1&0&0&0&0&0&0&0 \\
    1 &0&0&0&0&0&0&1&0&1&1&0&0&0&0&0 \\
    \hline
    0 &1&0&1&0 &0&0&0&0&0&0 &1&0&1&0 &0\\
    0 &1&0&1&0 &0&0&0&0&0&0 &1&0&1&0 &0\\
    0 &1&0&0&1 &0&0&0&0&0&0 &1&0&0&1 &0\\
    0 &0&1&1&0 &0&0&0&0&0&0 &0&1&1&0 &0\\
    0 &0&1&0&1 &0&0&0&0&0&0 &0&1&0&1 &0\\
    0 &0&1&0&1 &0&0&0&0&0&0 &0&1&0&1 &0\\
    \hline
    0 &0&0&0&0&1&1&1&0&0&0&0&0&0&0&1 \\
    0 &0&0&0&0&0&0&0&1&1&1&0&0&0&0&1 \\
    0 &0&0&0&0&1&1&0&1&0&0&0&0&0&0&1 \\
    0 &0&0&0&0&0&0&1&0&1&1&0&0&0&0&1 \\
    \hline
    0 &0&0&0&0 &0&0&0&0&0&0 &1&1&1&1 &0
    
\end{array}\right]\]
describes a $4$-regular (loopless, symmetric, unweighted) graph with an equitable distance partition, the same divisor matrix as the $4$-hypercube, and hence the same quantum state transfer as the hypercube occurring at the same time $t$. The same holds for any graph of this form where the $4\times6$ sub-matrices have three ones per column and two per row. This gives rise to at least four non-isomorphic graphs.

\subsection{Weighted Examples}

Let $p\ps{x}=\frac{x^{15}+3x^{11}+3x^7+x^3}{8}$ and define $U'\ps{t}=\bbm{p\ps{\cos{t}}&p\ps{i\sin\ps{t}}\\p\ps{i\sin\ps{t}}&p\ps{\cos\ps{t}}}$. It may be observed that $U'\ps{t}$ satisfies several conditions required of a submatrix of a quantum walk matrix (such as having rows and columns with norm at most $1$ for all $t\in\R$), and that when treated as such it exhibits quantum state transfer between the first and second rows/columns at time $t=\frac{\pi}{2}$. Using the quantum walk/reduction equivalence (Theorem \ref{Quantum walk equivalence to reduction}) and unfolding formulas above we may numerically approximate a matrix $A$ such that the restriction of $\bs{e^{-itA}}$ to its first two rows and columns equals $U\ps{t}$ for all $t$.\footnote{Code to accomplish this procedure can be found at \href{https://github.com/JakkobMath/Isospectral-reduction-code}{https://github.com/JakkobMath/Isospectral-reduction-code}.} For this particular example, we get \newline\noindent\resizebox{\textwidth}{!}{$\displaystyle A=\left(
\begin{array}{cccccccccccccccc}
 0. & 0. & 0.027978 & -0.0922718 & 0.242288 & -0.486778 & 0.823341 & -1.16907 & 1.46643 & -1.46643 & 1.16907 & -0.823341 & 0.486778 & -0.242288 & 0.0922718 & 0.027978 \\
 0. & 0. & 0.027978 & 0.0922718 & 0.242288 & 0.486778 & 0.823341 & 1.16907 & 1.46643 & 1.46643 & 1.16907 & 0.823341 & 0.486778 & 0.242288 & 0.0922718 & -0.027978 \\
 0.027978 & 0.027978 & -14.9999 & 0. & 0. & 0. & 0. & 0. & 0. & 0. & 0. & 0. & 0. & 0. & 0. & 0. \\
 -0.0922718 & 0.0922718 & 0. & -12.9986 & 0. & 0. & 0. & 0. & 0. & 0. & 0. & 0. & 0. & 0. & 0. & 0. \\
 0.242288 & 0.242288 & 0. & 0. & -10.9883 & 0. & 0. & 0. & 0. & 0. & 0. & 0. & 0. & 0. & 0. & 0. \\
 -0.486778 & 0.486778 & 0. & 0. & 0. & -8.94012 & 0. & 0. & 0. & 0. & 0. & 0. & 0. & 0. & 0. & 0. \\
 0.823341 & 0.823341 & 0. & 0. & 0. & 0. & -6.77289 & 0. & 0. & 0. & 0. & 0. & 0. & 0. & 0. & 0. \\
 -1.16907 & 1.16907 & 0. & 0. & 0. & 0. & 0. & -4.37061 & 0. & 0. & 0. & 0. & 0. & 0. & 0. & 0. \\
 1.46643 & 1.46643 & 0. & 0. & 0. & 0. & 0. & 0. & -1.54827 & 0. & 0. & 0. & 0. & 0. & 0. & 0. \\
 -1.46643 & 1.46643 & 0. & 0. & 0. & 0. & 0. & 0. & 0. & 1.54827 & 0. & 0. & 0. & 0. & 0. & 0. \\
 1.16907 & 1.16907 & 0. & 0. & 0. & 0. & 0. & 0. & 0. & 0. & 4.37061 & 0. & 0. & 0. & 0. & 0. \\
 -0.823341 & 0.823341 & 0. & 0. & 0. & 0. & 0. & 0. & 0. & 0. & 0. & 6.77289 & 0. & 0. & 0. & 0. \\
 0.486778 & 0.486778 & 0. & 0. & 0. & 0. & 0. & 0. & 0. & 0. & 0. & 0. & 8.94012 & 0. & 0. & 0. \\
 -0.242288 & 0.242288 & 0. & 0. & 0. & 0. & 0. & 0. & 0. & 0. & 0. & 0. & 0. & 10.9883 & 0. & 0. \\
 0.0922718 & 0.0922718 & 0. & 0. & 0. & 0. & 0. & 0. & 0. & 0. & 0. & 0. & 0. & 0. & 12.9986 & 0. \\
 0.027978 & -0.027978 & 0. & 0. & 0. & 0. & 0. & 0. & 0. & 0. & 0. & 0. & 0. & 0. & 0. & 14.9999 \\
\end{array}
\right).$}

Now we apply the correctness-preserving similarities to clean up this matrix. Call the submatrix of $A$ consisting of its first two columns and all but the last two rows $C$ and let $C$ have a singular-value decomposition as $C=U\Sigma V^*$. Conjugating $A$ by $\bbm{I_2&0\\0&V}$ gives us the matrix
\newline\noindent\resizebox{\textwidth}{!}{$\displaystyle A'=\left(
\begin{array}{cccccccccccccccc}
 0. & 0. & 2.12132 & 2.12132 & 0. & 0. & 0. & 0. & 0. & 0. & 0. & 0. & 0. & 0. & 0. & 0. \\
 0. & 0. & 2.12132 & -2.12132 & 0. & 0. & 0. & 0. & 0. & 0. & 0. & 0. & 0. & 0. & 0. & 0. \\
 2.12132 & 2.12132 & -0.0833333 & 0. & 0.196751 & 0.284463 & -2.57647 & 0.68318 & -0.977199 & 0.856952 & 2.48288 & 0.481143 & 2.08239 & 0.141588 & 0.571262 & -0.0163498 \\
 2.12132 & -2.12132 & 0. & 0.0833333 & 0. & 2.05577 & 0.219783 & 2.41895 & 0.39145 & -1.05739 & 0.312072 & -2.6215 & 0.129941 & -1.2529 & 0.024631 & 0.197587 \\
 0. & 0. & 0.196751 & 0. & -14.9973 & 0.0037521 & -0.033984 & 0.0090112 & -0.0128894 & 0.0113033 & 0.0327495 & 0.0063463 & 0.0274669 & 0.0018676 & 0.007535 & -0.0002157 \\
 0. & 0. & 0.284463 & 2.05577 & 0.0037521 & -8.5725 & -0.0047034 & 0.307882 & 0.0858303 & -0.548067 & 0.15351 & -0.770738 & 0.0912691 & -0.336743 & 0.0219053 & 0.0509662 \\
 0. & 0. & -2.57647 & 0.219783 & -0.033984 & -0.0047034 & -6.7053 & -0.0536042 & -1.27472 & -0.13598 & -2.27623 & -0.110415 & -1.35281 & -0.0405812 & -0.324624 & 0.005575 \\
 0. & 0. & 0.68318 & 2.41895 & 0.0090112 & 0.307882 & -0.0536042 & -5.01216 & 0.13078 & -3.04851 & 0.308603 & -2.82363 & 0.194183 & -1.09495 & 0.0478674 & 0.155452 \\
 0. & 0. & -0.977199 & 0.39145 & -0.0128894 & 0.0858303 & -1.27472 & 0.13078 & -6.30342 & 0.0416119 & -6.03507 & -0.0373148 & -3.23428 & -0.0253876 & -0.734529 & 0.0045148 \\
 0. & 0. & 0.856952 & -1.05739 & 0.0113033 & -0.548067 & -0.13598 & -3.04851 & 0.0416119 & -5.09013 & 0.289492 & -5.12205 & 0.202933 & -1.83847 & 0.052339 & 0.24869 \\
 0. & 0. & 2.48288 & 0.312072 & 0.0327495 & 0.15351 & -2.27623 & 0.308603 & -6.03507 & 0.289492 & -2.22977 & 0.114164 & -3.32338 & 0.02211 & -0.726791 & -0.001291 \\
 0. & 0. & 0.481143 & -2.6215 & 0.0063463 & -0.770738 & -0.110415 & -2.82363 & -0.0373148 & -5.12205 & 0.114164 & 3.11392 & 0.0937966 & -1.26269 & 0.0255681 & 0.166242 \\
 0. & 0. & 2.08239 & 0.129941 & 0.0274669 & 0.0912691 & -1.35281 & 0.194183 & -3.23428 & 0.202933 & -3.32338 & 0.0937966 & 7.31687 & 0.0228196 & -0.348013 & -0.0021095 \\
 0. & 0. & 0.141588 & -1.2529 & 0.0018676 & -0.336743 & -0.0405812 & -1.09495 & -0.0253876 & -1.83847 & 0.02211 & -1.26269 & 0.0228196 & 10.562 & 0.0066175 & 0.0552393 \\
 0. & 0. & 0.571262 & 0.024631 & 0.007535 & 0.0219053 & -0.324624 & 0.0478674 & -0.734529 & 0.052339 & -0.726791 & 0.0255681 & -0.348013 & 0.0066175 & 12.925 & -0.0006645 \\
 0. & 0. & -0.0163498 & 0.197587 & -0.0002157 & 0.0509662 & 0.005575 & 0.155452 & 0.0045148 & 0.24869 & -0.001291 & 0.166242 & -0.0021095 & 0.0552393 & -0.0006645 & 14.9928 \\
\end{array}
\right)$.}

Because this is a correctness-preserving similarity, $A'$ has the same reduction (to its first two rows and columns) as $A$, and hence the same restricted quantum walk matrix as $A$. Now let $C'$ be the submatrix of $A'$ consisting of the third and fourth rows and the fifth through final columns, and let $V'$ be the right unitary matrix in the singular value decomposition for $C'$. Conjugating $A'$ by $\bbm{I_4&0\\0&V'}$ gets us a new matrix $A''$ with the same restricted quantum walk matrix as $A$. By iterating this process (and applying some similarities by block-diagonal matrices with two-by-two unitary matrices such as $\bbm{1&0\\0&1}$ and $\bbm{\frac{\sqrt{2}}{2}&\frac{\sqrt{2}}{2}\\\frac{\sqrt{2}}{2}&-\frac{\sqrt{2}}{2}}$ on the diagonal) we may find the block-tridiagonal matrix 
\newline\noindent\resizebox{\textwidth}{!}{$\displaystyle A''=\left(
\begin{array}{cccccccccccccccc}
 0. & 0. & -3. & 0. & 0. & 0. & 0. & 0. & 0. & 0. & 0. & 0. & 0. & 0. & 0. & 0. \\
 0. & 0. & 0. & -3. & 0. & 0. & 0. & 0. & 0. & 0. & 0. & 0. & 0. & 0. & 0. & 0. \\
 -3. & 0. & 0. & -0.0833333 & -4.47136 & 0. & 0. & 0. & 0. & 0. & 0. & 0. & 0. & 0. & 0. & 0. \\
 0. & -3. & -0.0833333 & 0. & 0. & -4.47136 & 0. & 0. & 0. & 0. & 0. & 0. & 0. & 0. & 0. & 0. \\
 0. & 0. & -4.47136 & 0. & 0. & 0.200041 & -5.56723 & 0. & 0. & 0. & 0. & 0. & 0. & 0. & 0. & 0. \\
 0. & 0. & 0. & -4.47136 & 0.200041 & 0. & 0. & 5.56723 & 0. & 0. & 0. & 0. & 0. & 0. & 0. & 0. \\
 0. & 0. & 0. & 0. & -5.56723 & 0. & 0. & 0.619732 & 6.40559 & 0. & 0. & 0. & 0. & 0. & 0. & 0. \\
 0. & 0. & 0. & 0. & 0. & 5.56723 & 0.619732 & 0. & 0. & 6.40559 & 0. & 0. & 0. & 0. & 0. & 0. \\
 0. & 0. & 0. & 0. & 0. & 0. & 6.40559 & 0. & 0. & -1.31413 & 7.07649 & 0. & 0. & 0. & 0. & 0. \\
 0. & 0. & 0. & 0. & 0. & 0. & 0. & 6.40559 & -1.31413 & 0. & 0. & 7.07649 & 0. & 0. & 0. & 0. \\
 0. & 0. & 0. & 0. & 0. & 0. & 0. & 0. & 7.07649 & 0. & 0. & 3.17371 & -7.21365 & 0. & 0. & 0. \\
 0. & 0. & 0. & 0. & 0. & 0. & 0. & 0. & 0. & 7.07649 & 3.17371 & 0. & 0. & -7.21365 & 0. & 0. \\
 0. & 0. & 0. & 0. & 0. & 0. & 0. & 0. & 0. & 0. & -7.21365 & 0. & 0. & -5.45303 & -7.37469 & 0. \\
 0. & 0. & 0. & 0. & 0. & 0. & 0. & 0. & 0. & 0. & 0. & -7.21365 & -5.45303 & 0. & 0. & 7.37469 \\
 0. & 0. & 0. & 0. & 0. & 0. & 0. & 0. & 0. & 0. & 0. & 0. & -7.37469 & 0. & 0. & -11.0904 \\
 0. & 0. & 0. & 0. & 0. & 0. & 0. & 0. & 0. & 0. & 0. & 0. & 0. & 7.37469 & -11.0904 & 0. \\
\end{array}
\right).$}

By construction, $A''$ has the same reduction to its first two rows and columns as $A$. Finally, we conjugate $A''$ by the diagonal matrix with diagonal entries $$1, 1, -1, -1, 1, 1, -1, 1, -1, 1, -1, 1, 1, -1, -1, -1$$ to obtain the matrix \newline\noindent\resizebox{\textwidth}{!}{$\displaystyle A'''=\left(
\begin{array}{cccccccccccccccc}
 0. & 0. & 3. & 0. & 0. & 0. & 0. & 0. & 0. & 0. & 0. & 0. & 0. & 0. & 0. & 0. \\
 0. & 0. & 0. & 3. & 0. & 0. & 0. & 0. & 0. & 0. & 0. & 0. & 0. & 0. & 0. & 0. \\
 3. & 0. & 0. & -0.0833333 & 4.47136 & 0. & 0. & 0. & 0. & 0. & 0. & 0. & 0. & 0. & 0. & 0. \\
 0. & 3. & -0.0833333 & 0. & 0. & 4.47136 & 0. & 0. & 0. & 0. & 0. & 0. & 0. & 0. & 0. & 0. \\
 0. & 0. & 4.47136 & 0. & 0. & 0.200041 & 5.56723 & 0. & 0. & 0. & 0. & 0. & 0. & 0. & 0. & 0. \\
 0. & 0. & 0. & 4.47136 & 0.200041 & 0. & 0. & 5.56723 & 0. & 0. & 0. & 0. & 0. & 0. & 0. & 0. \\
 0. & 0. & 0. & 0. & 5.56723 & 0. & 0. & -0.619732 & 6.40559 & 0. & 0. & 0. & 0. & 0. & 0. & 0. \\
 0. & 0. & 0. & 0. & 0. & 5.56723 & -0.619732 & 0. & 0. & 6.40559 & 0. & 0. & 0. & 0. & 0. & 0. \\
 0. & 0. & 0. & 0. & 0. & 0. & 6.40559 & 0. & 0. & 1.31413 & 7.07649 & 0. & 0. & 0. & 0. & 0. \\
 0. & 0. & 0. & 0. & 0. & 0. & 0. & 6.40559 & 1.31413 & 0. & 0. & 7.07649 & 0. & 0. & 0. & 0. \\
 0. & 0. & 0. & 0. & 0. & 0. & 0. & 0. & 7.07649 & 0. & 0. & -3.17371 & 7.21365 & 0. & 0. & 0. \\
 0. & 0. & 0. & 0. & 0. & 0. & 0. & 0. & 0. & 7.07649 & -3.17371 & 0. & 0. & 7.21365 & 0. & 0. \\
 0. & 0. & 0. & 0. & 0. & 0. & 0. & 0. & 0. & 0. & 7.21365 & 0. & 0. & 5.45303 & 7.37469 & 0. \\
 0. & 0. & 0. & 0. & 0. & 0. & 0. & 0. & 0. & 0. & 0. & 7.21365 & 5.45303 & 0. & 0. & 7.37469 \\
 0. & 0. & 0. & 0. & 0. & 0. & 0. & 0. & 0. & 0. & 0. & 0. & 7.37469 & 0. & 0. & -11.0904 \\
 0. & 0. & 0. & 0. & 0. & 0. & 0. & 0. & 0. & 0. & 0. & 0. & 0. & 7.37469 & -11.0904 & 0. \\
\end{array}
\right).$}

This matrix has mostly positive entries, has the same reduction to its first two vertices as $A$, has the $U'\ps{t}$ defined above as the restriction of its quantum walk matrix to its first two rows and columns, and has no loops. Its eigenvalues are the odd integers between $-15$ and $15$ inclusive, each with multiplicity one. It exhibits perfect state transfer between its first two vertices and is the matrix of minimal size having $U'\ps{t}$ as a restriction of its quantum walk matrix. Furthermore, any matrix having $U'\ps{t}$ as a submatrix of its quantum walk matrix may be obtained from $A'''$ by appending new rows and columns such that the existing eigenvalues are preserved and the restrictions of the eigenvectors of the new matrix to its first $16$ rows and columns match the eigenvectors of $A'''$, applying conjugations by invertible matrices that equal the identity matrix on their first two rows and columns, and applying conjugations by permutation matrices.


\bibliographystyle{plain}
\bibliography{Citations}

\begin{thebibliography}{10}

\bibitem{Brauer_1947}
A.~Brauer.
\newblock Limits for the characteristic roots of a matrix ii.
\newblock {\em Duke Math. J.}, 14:21--26, 1947.

\bibitem{Brualdi_1982}
R.~Brualdi.
\newblock Matrices, eigenvalues, and directed graphs.
\newblock {\em Lin. Multilin. Alg.}, 111:143--165, 1982.

\bibitem{Bunimovich_2011}
L.~A. Bunimovich and B.~Z. Webb.
\newblock Isospectral graph transformations, spectral equivalence, and global
  stability of dynamical networks.
\newblock {\em Nonlinearity}, 25(1):211--254, dec 2011.

\bibitem{Bunimovich_2013}
L.~A. Bunimovich and B.~Z. Webb.
\newblock Restrictions and stability of time-delayed dynamical networks.
\newblock {\em Nonlinearity}, 26(8):2131--2156, 2013.

\bibitem{bunimovich_2014}
L.~A. Bunimovich and B.~Z. Webb.
\newblock {\em Improved Estimates of Survival Probabilities via Isospectral
  Transformations. In: W. Bahsoun, C. Bose, and G. Froyland (eds) Ergodic
  Theory, Open Dynamics, and Coherent Structures}.
\newblock Springer New York, 2014.

\bibitem{MR3237552}
Leonid Bunimovich and Benjamin Webb.
\newblock {\em Isospectral transformations}.
\newblock Springer Monographs in Mathematics. Springer, New York, 2014.
\newblock A new approach to analyzing multidimensional systems and networks.

\bibitem{chan2022fundamentals}
Ada Chan, Gabriel Coutinho, Whitney Drazen, Or~Eisenberg, Chris Godsil, Mark
  Kempton, Gabor Lippner, Christino Tamon, and Hanmeng Zhan.
\newblock Fundamentals of fractional revival in graphs.
\newblock {\em Linear Algebra and its Applications}, 655:129--158, 2022.

\bibitem{chan2019quantum}
Ada Chan, Gabriel Coutinho, Christino Tamon, Luc Vinet, and Hanmeng Zhan.
\newblock Quantum fractional revival on graphs.
\newblock {\em Discrete Applied Mathematics}, 269:86--98, 2019.

\bibitem{chan2021pretty}
Ada Chan, Whitney Drazen, Or~Eisenberg, Mark Kempton, and Gabor Lippner.
\newblock Pretty good quantum fractional revival in paths and cycles.
\newblock {\em Algebraic Combinatorics}, 4(6):989--1004, 2021.

\bibitem{doetsch2012introduction}
Gustav Doetsch.
\newblock {\em Introduction to the Theory and Application of the Laplace
  Transformation}.
\newblock Springer Science \& Business Media, 2012.

\bibitem{duarte2015eigenvectors}
Pedro Duarte and Maria~Joana Torres.
\newblock Eigenvectors of isospectral graph transformations.
\newblock {\em Linear Algebra and its Applications}, 474:110--123, 2015.

\bibitem{Gershgorin_1931}
S.~Gershgorin.
\newblock \"{U}ber die abgrenzung der eigenwerte einer matrix.
\newblock {\em Izv. Akad. Nauk SSSR Ser. Mat.}, 1:749--754, 1931.

\bibitem{godsil_survey}
Chris Godsil.
\newblock State transfer on graphs.
\newblock {\em Discrete Math.}, 312(1):129--147, 2012.

\bibitem{godsil2012number}
Chris Godsil, Stephen Kirkland, Simone Severini, and Jamie Smith.
\newblock Number-theoretic nature of communication in quantum spin systems.
\newblock {\em Physical review letters}, 109(5):050502, 2012.

\bibitem{godsil2001algebraic}
Chris Godsil and Gordon~F Royle.
\newblock {\em Algebraic graph theory}, volume 207.
\newblock Springer Science \& Business Media, 2001.

\bibitem{kay2010perfect}
Alastair Kay.
\newblock Perfect, efficient, state transfer and its application as a
  constructive tool.
\newblock {\em International Journal of Quantum Information}, 8(04):641--676,
  2010.

\bibitem{kempton2020characterizing}
Mark Kempton, John Sinkovic, Dallas Smith, and Benjamin Webb.
\newblock Characterizing cospectral vertices via isospectral reduction.
\newblock {\em Linear Algebra and its Applications}, 594:226--248, 2020.

\bibitem{Reber_2019}
D.~Reber and B.~Z. Webb.
\newblock Intrinsic stability: Global stability of dynamical networks and
  switched systems resilient to any type of time-delays.
\newblock {\em submitted to Nonlinearity}, 2019.

\bibitem{rontgen2020designing}
M~R{\"o}ntgen, NE~Palaiodimopoulos, CV~Morfonios, I~Brouzos, M~Pyzh,
  FK~Diakonos, and P~Schmelcher.
\newblock Designing pretty good state transfer via isospectral reductions.
\newblock {\em Physical Review A}, 101(4):042304, 2020.

\bibitem{Vasquez_2015}
F.~Guevara Vasquez and B.~Z. Webb.
\newblock Pseudospectra of isospectrally reduced matrices.
\newblock {\em Numer. Linear Algebra with Appl.}, 22(1):145--174, 2015.

\end{thebibliography}

\end{document}